\documentclass{amsart}

\newcommand{\al}{\alpha}
\newcommand{\be}{\beta}
\newcommand{\ga}{\gamma}
\newcommand{\de}{\delta}
\newcommand{\f}{\phi}
\newcommand{\ff}{\Phi}

\newcommand{\PP}{\mathbb{P}} 
\newcommand{\A}{\mathbb{A}}
\newcommand{\G}{\mathbb{G}}

\newcommand{\B}{\mathcal{B}}
\newcommand{\T}{\mathcal{T}}
\renewcommand{\H}{\mathcal{H}}
\newcommand{\R}{\mathcal{R}}
\renewcommand{\O}{\mathcal{O}}

\newcommand{\p}{\mathfrak{p}}
\newcommand{\C}{\mathfrak{c}}

\newcommand{\ol}{\overline}    
\newcommand{\og}{\ol g}
\newcommand{\ox}{\ol x}
\newcommand{\oy}{\ol y}
\newcommand{\oK}{\ol K}
\newcommand{\op}{\ol{\mathfrak p}}
\newcommand{\oj}{\ol{\jmath}}

\newcommand{\vF}{\check F}
\newcommand{\vx}{\check x}
\newcommand{\vy}{\check y}
\newcommand{\va}{\check a}
\newcommand{\vb}{\check b}

\newcommand{\vT}{T'}
\newcommand{\vB}{B'}
\newcommand{\vf}{\f'}

\newcommand{\uW}{\breve W}
\newcommand{\uu}{\breve u}
\newcommand{\uv}{\breve v}
\newcommand{\ux}{\breve x}
\newcommand{\uy}{\breve y}

\DeclareMathOperator {\rank}{rank}
\def\Spec{\operatorname{Spec}}
\def\Div{\operatorname{div}}
\def\ord{\operatorname{ord}}

\newtheorem{thm}{Theorem}[section]

\newtheorem{cor}[thm]{Corollary}
\newtheorem{prop}[thm]{Proposition}

\begin{document}
\title[Fibrations by non-smooth curves of arithmetic genus two]
{Fibrations by non-smooth projective curves of\\
arithmetic genus two in characteristic two}

\author{Alejandro Simarra Ca\~nate}
\address{Universidade Federal Fluminense -- Instituto de Matem\'atica,
Rua Mario Santos Braga s/n, 24020-140 -- Niter\'oi --RJ, Brazil}

\email{alsica@impa.br}


\author{Karl-Otto St\"ohr}
\address{Instituto de Matem\'atica Pura e Aplicada (IMPA),
Estrada Dona Castorina 110, 22460-320 -- Rio de Janeiro -- RJ, Brazil}
\email{stohr@impa.br}


\subjclass[2010]{Primary 14D06, 14E05, 14H10, 14H20, 14H52}

\keywords{Bertini--Sard theorem, fibrations by non-smooth curves, generic fibre,
geometrical generic fibre, elliptic curves,
non-conservative function fields}

\date{ \today }

\begin{abstract}
Looking in positive characteristic for failures of the Bertini--Sard
theorem, we determine, up to birational equivalence, the separable
proper morphisms of smooth algebraic varieties in characteristic two,
whose fibres are non-smooth curves of arithmetic genus two.

\end{abstract}

\maketitle

\section*{Introduction}
\label{I}
\noindent
Bertini's theorem on moving singularities, published in the last but one
decade of the nineteens century, has become a fundamental tool in Algebraic
Geometry. Nowadays, due to its similarities to Sard's theorem on differentiable
maps, it is also called the Bertini--Sard theorem. It assures that almost all fibres
of a dominant morphism between smooth algebraic varieties are smooth.

However, in the 1940s Zariski~\cite{Z1} observed that the theorem may fail in positive
characteristic. He had constructed a fibration \,$\f:T\rightarrow B$\,
by algebraic curves that admits moving singularities, though the total space $T$
is smooth. A moving singularity of
\,$\f$\, can be viewed as a horizontal prime divisor on \,$T$\, with the
property that each of its points is a singular point of the fibre to which
it belongs.

Translated in modern language, Zariski argued that, though the generic fibre
is a regular scheme over the base field \,$k(B)$\,, it may not be smooth,
that is, the geometric generic fibre, defined by extending the base field
to its algebraic closure \,$\overline{k(B)}$, may have singularities.
This means that the function
field \,$k(T)|k(B)$\, may be non-conservative, that is, its genus $g$ may
decrease by tensoring with \,$\overline{k(B)}$.
For more explications we refer to Section~\ref{A}.

To rescue Bertini's theorem in positive characteristic \,$p$\,, we are conduced
to classify its exceptions. As follows from a theorem of Tate~\cite{T1},
Bertini's theorem can only fail if \,$p\leq 2\,g+1$\,. Non-conservative
function fields of genus $1$ were classified by Queen~\cite{Q}, and of
genus $2$ in odd characteristic by Borges Neto~\cite{Bo}. General results on
non-conservative function fields and their singular primes were developed
by Stichtenoth, Bedoya and the second author
in the papers \cite{Sti}, \cite{BS} and \cite{St1}.

In their program to extend Enriques' classification of algebraic   
surfaces to arbitrary characteristic, Bombieri and Mumford~\cite{BM}
encountered \textit{quasi-elliptic fibrations}, i.e., fibrations by cuspidal
curves of arithmetic genus $1$\,.
There is a large number of recent papers on classification theory of
algebraic varieties and singularities in positive characteristic, too large
to put in our references, which can be found by starting the search with
\cite{BM} and looking successively for citations and references.
Singularities of generic fibres in positive characteristic
were analyzed by Schr\"oer~\cite{Sc}.
Fibrations by non-smooth curves of arithmetic genus $3$ in characteristic
$3$, $5$ and $7$ were studied by Salom\~ao~\cite{Sa1},\cite{Sa2}
and the second author \cite{St3},\cite{St4}.

\medskip
In the present paper we study in characteristic two the fibrations by
non-smooth curves of arithmetic genus two.
We realize the fibres by tri-canonical embeddings
as curves on a cone in $\PP^4$.
The discussion
naturally divides into two cases. If the function field of the generic fibre
is separable over its canonical quadratic rational subfield, then we prove
that almost every fibre is geometrically elliptic, i.e., its non-singular
model is an elliptic curve (see Theorem~\ref{C1}). In the second case
the fibres are rational, as discussed in Theorem~\ref{D1}.

We discover a $6$-dimensional smooth algebraic variety
\,$Z \subset \PP^4\times\A^5$\, such that almost all fibres of the
projection morphism \,$\pi :Z\rightarrow \A^5$\, are cuspidal
geometrically elliptic curves of arithmetic genus two (see Theorem~\ref{E1}).
We describe how the elliptic modular invariant of the fibres varies,
determine the singular points of all fibres, and discuss how the singularities move.

Theorem~\ref{E2} is the main result
of this paper. It states that each proper separable morphism
between smooth algebraic varieties, whose fibres are geometrically elliptic
curves of arithmetic genus two, is birational equivalent to a base
extension of the fibration \,$\pi :Z\rightarrow \A^5$\,.
A similar result for fibrations by rational curves of arithmetic genus
two is also obtained (see Theorem~\ref{E4}).

\section{Moving singularities of fibrations by algebraic curves}
\label{A}
\noindent
In this introductory section we present prerequisites on moving singularities
of fibrations by algebraic curves, needed to understand our paper.

Let $\f:T \rightarrow B$ be a dominant morphism of irreducible
algebraic varieties defined over an algebraically closed field $k$.
We assume that $\dim T = \dim B + 1$\, or, equivalently, almost all fibres
are algebraic curves (see \cite[p.\ 74]{Sh}). Thus by restricting if
necessary the base variety $B$ to a dense open subvariety we get a fibration
by algebraic curves.

By identifying the rational functions on the base $B$ with rational functions
on the total space $T$ that are constant along each fibre, we can view the
field $k(B)$ of the base as a subfield of $k(T)$. We assume that almost all
fibres are integral. By a theorem of Matsusaka this means that $k(B)$ is
algebraically closed in $k(T)$ and that $k(T)$ is separable over $k(B)$
(see \cite{M}, \cite[pp.\ 256--257]{Sh}). Thus the field $k(T)$ of the total
space, which is a higher dimensional function field over the
constant field $k$, becomes a one-dimensional separable function field over
the base field $ k(B)$. In this sense, the fibrations by integral algebraic
curves over the variety $B$, up to birational equivalence, correspond
bijectively to the isomorphism classes of the one-dimensional separable
function fields over $k(B)$.

In the setting of schemes, the function field $k(T)|k(B)$ is the field of
the \textit{generic fibre} $\T\times_\B\, \Spec k(B)$\,, where the
calligraphic letters $\T$ and $\B$ stand for the integral schemes whose points
correspond bijectively to the closed irreducible subsets of $T$ and $B$,
respectively. The generic fibre is a geometrically integral curve over
$k(B)$. Its closed points, which are exactly its non-generic points,
correspond bijectively to the \textit{horizontal prime divisors} of the
fibration $\f:T\rightarrow B\,$, that is, to the prime divisors of the total
space $T$ whose images are dense in the base variety $B$\,. The local ring of
the scheme $\T$ (and also the local ring of the generic fibre) at a closed
point of the generic fibre is equal to the local ring of the total space
$T$ along the corresponding horizontal prime divisor, and its residue field
is isomorphic to the field of rational functions on the horizontal prime
divisor.

As the \textit{non-smooth locus} of the morphism $\f:T\rightarrow B$ (i.e.,
the union of the non-smooth loci of the fibres of $\f$) is closed in $T$,
and as even the non-smooth locus of the corresponding morphism 
$\ff:\T \rightarrow \B$ of schemes is closed in $\T$
(cf.\ \cite[p.\ 224, Cor.\ 2.12]{L}), we deduce that a closed irreducible
subset $H$ of $T$ is contained in the non-smooth locus of $\f$ if and only
if the corresponding integral scheme $\H$\, or, equivalently, its generic
point is contained in the non-smooth locus of $\ff$\,. Applying this
remark to the horizontal prime divisors, we obtain:

\begin{prop}
\label{A1}
The horizontal prime divisors contained in the non-smooth locus of the
fibration $\f:T\rightarrow B$ correspond bijectively  to the non-smooth
closed points of the generic fibre $\T\times_\B\,\Spec k(B)$.
\end{prop}

These horizontal prime divisors, whose points are singular points
of the fibres to which they belong, are called the
\textit{moving singularities} of the fibration.
Here we do not consider singularities that move in subvarieties of codimension
larger than $1$\,.

We always assume that the dominant morphism $\f:T\rightarrow B$ is proper. This
implies that its fibres are complete, that $\f$ is surjective and that even
the restrictions of $\f$ to the horizontal prime divisors are surjective.
We further assume that the total space $T$ is smooth. In particular, it is
regular in codimension one. Thus the generic fibre 
is a regular complete geometrically integral algebraic curve over $k(B)$,
or more precisely,
\[
\T\times_\B\,\Spec k(B) = \R_{k(T)|k(B)}
\]
where $\R_{k(T)|k(B)}$ denotes the \textit{regular complete model} of the
one-dimensional function field $k(T)|k(B)$. The closed points of
$\R_{k(T)|k(B)}$ are exactly the primes $\p$ of $k(T)|k(B)$, and their
local rings are the corresponding discrete valuation rings $\O_\p$\,.
By Proposition~\ref{A1} a horizontal prime divisor is a moving singularity
if and only if the corresponding prime $\p$ is a \textit{singular prime}
in the sense that $\p$ is a non-smooth point of $\R_{k(T)|k(B)}$, i.e.,
the one-dimensional semi-local ring $\O_\p\otimes_{k(B)}\,\ol{k(B)}$
is non-regular, i.e., over the point $\p$ of the generic fibre there lies a
singular point of the \textit{geometric generic fibre}
\begin{equation*}
\begin{aligned}
\T\times_\B\,\Spec\ol{k(B)}
&= (\T\times_\B\,\Spec k(B))\times_{\Spec k(B)}\,\Spec\ol{k(B)}
\\ 
& = \R_{k(T)|k(B)}\times_{\Spec k(B)}\,\Spec\ol{k(B)} \,.
\end{aligned}
\end{equation*}
By Rosenlicht's genus drop formula
(see Section~\ref{B})
the number of the
singular primes of $k(T)|k(B)$, counted according to their singularity degrees,
is equal to $g-\og$, where $g$ and $\og$ denote the genera of the function
fields $k(T)|k(B)$ and $k(T)\otimes_{k(B)}\ol{k(B)}\mid\ol{k(B)}\,,$
respectively. 
Thus the fibration $\f:T\rightarrow B$ admits moving
singularities if and only if $\og<g$\,. 
As the genus remains invariant under separable base field extensions,
we obtain \textit{Bertini's theorem on moving singular points}.

\medskip
\noindent
\textbf{Bertini--Sard Theorem}.\
\textit{In characteristic zero the fibration \mbox{$\f:T\rightarrow B$} 
does not admit moving singularities, i.e., almost all fibres are smooth.}

\medskip
Moreover, if the characteristic is a prime $p$, then by a theorem of
Tate~\cite{T1} the genus drop
$g-\og$ is a multiple of $\frac{p-1}{2}$\,,
and so Bertini's theorem can only fail if $p\leq 2g+1$.

By restricting the base $B$ to a dense open subvariety, we may assume that
all fibres are of dimension one, and that not only the total space $T$ but
also the base $B$ is smooth. Then the morphism $\f:T\rightarrow B$ is flat
(see \cite[Theorem 18.16]{E}), and so the arithmetic genus of each fibre
is equal to the arithmetic genus of the generic fibre
$\T\times_{\B}\,\Spec k(B)$ (see \cite[Ch.\ \textrm{III}, Theorem 9.9]{Ha}).
As the generic fibre is equal to $\R_{k(T)|k(B)}$\,, its arithmetic genus
is equal to the genus $g$ of the function field $k(T)|k(B)$. Moreover,
as the arithmetic genus is invariant under base field extensions, the
genus $g$ is also equal to the arithmetic genus of the geometric generic
fibre $\R_{k(T)|k(B)}\times_{\Spec k(B)}\,\Spec \ol{k(B)}$\,.

The geometric generic fibre is a complete integral curve over $\ol{k(B)}$
of geometric genus $\og$. It reflects the properties of the closed fibres
in a better way than the generic fibre. By semi-continuity, the geometric
genus of each closed fibre is smaller than or equal to $\og$, and equality
holds for almost all fibres.

\section{Curves of arithmetic genus $2$ on a cone in $\PP^4$}
\label{B}
\noindent
As becomes clear from the preceding section, we can apply the theory of
function fields in order to study the generic fibres of morphisms.

Let $F|K$ be a one-dimensional function field of genus $g=2$\,, and let
$\C$ be a canonical divisor of $F|K$. By the Riemann--Roch theorem its
degree and the dimension of its space of global sections are equal to
\[
\deg(\C) = 2g-2 =2 \quad \text{and} \quad\dim H^0(\C) =g = 2 \,.
\]
As $\,H^0(\C)\neq 0\,,$ the canonical divisor $\C$ is linearly equivalent
to a positive divisor, and so we can assume that $\C$ is positive. Since
the dimension of $H^0(\C)$\, is larger than $1$\,, there is a function
$x \in H^0(\C) \setminus K\,.$ As $\,\Div_\infty(x) \leq \C$\,, the
fundamental equality $[F:K(x)] = \deg \Div_\infty(x)$ implies
$[F:K(x)] \leq \deg(\C) = 2$\,. 
Since $g \neq 0$\,, and therefore $F \neq K(x)$\,, we conclude that
\[
[F:K(x)] = 2 \quad \text{and} \quad \Div_\infty(x) = \C\,.
\]
As $\,H^0(\C) = K \oplus Kx$\,, the canonical positive divisors different
from $\C$ are just of the form
\[
\Div_0(x-a)\quad\text{where} \;\; a\in K\,,
\]
and the canonical subfield of $F|K$\,, i.e., the field generated by the
global sections of any canonical positive canonical divisor, is equal to
the rational quadratic subfield $K(x)$ of $F|K$\,. Clearly
\[
H^0(\C^n)\cap K(x)=Kx^0\oplus\dots\oplus Kx^n\quad\text{for each}\;\;n\geq 2\,.
\]
Moreover, by Riemann's theorem
\[
\dim H^0(\C^n) =2n-1 \quad \text{for each}\;\; n\geq 2 \,.
\]
Thus there is a function $y\in H^0(\C^3)\setminus K(x)$\,, and we obtain
\[
H^0(\C^n)=
{\textstyle
\bigoplus\limits_{i=0}^n Kx^i 
\, \oplus \,   
\bigoplus\limits_{i=0}^{n-3}Kx^iy
}
\quad\text{for each}\;\;n\geq 3\,.
\]
As $y^2 \in H^0(\C^6)$ there is an equation
\[
y^2 + a(x) y + b(x) = 0
\]
where $a(x)=\sum_{i=0}^3 a_i x^i$\, and
\,$b(x)=\sum_{i=0}^6 b_i x^i$\,
are polynomials with coefficients in $K$ of formal degree $3$ and $6$,
respectively. Since $[F:K(x)] = 2$ and $y\notin K(x)$ we have
\[
F = K(x,y)
\]
and the above equation is the minimal equation of $y$ over $K(x)$.

Let $\vF = K(\vx,\vy),$ where
\,$\vy^2+\va(\vx)\vy+\vb(\vx)=0$\,,
be another genus-$2$ function field in the above normal form.
As $K(\vx)$ and $K(x)$ are the canonical subfields of $\vF|K$ and
$F|K$\,, respectively, it is easy to check that the $K$-isomorphisms
$\vF\,\tilde {\rightarrow}\,F$
are just given by the transformations
\[
\vx\longmapsto
\frac{\al_{11}x+\al_{12}}{\al_{21}x+\al_{22}}
\quad\text{and}\quad \vy\longmapsto
\frac{\be\,y+\sum\nolimits_{i=0}^3\ga_i x^i}
{(\al_{21}x+\al_{22})^3}
\]
where $(\al_{ij}) \in \operatorname{GL}_2(K),\ \be \in K^*\,\text{ and }\,
\ga_0,\ga_1,\ga_2,\ga_3\in K$\, such that
\[     \! \! \! \! 
\be\,a(x) =
\left(\al_{21}x+\al_{22}\right)^3\,\,
\va\!\left(\frac{\al_{11}x+\al_{12}}{\al_{21}x+\al_{22}}\right)
  + 2\, \sum\nolimits_{i=0}^3 \ga_i x^i
\]
and
\begin{equation*} 
\begin{aligned}
\be^2 \, b(x)  = &
\left(\al_{21}x+\al_{22}\right)^6\,\,
\vb\!\left(\frac{\al_{11}x+\al_{12}}{\al_{21}x+\al_{22}}\right)
  + \left(\sum\nolimits_{i=0}^3 \ga_i x^i \right)^2
\\
&+\left(\al_{21}x+\al_{22}\right)^3\,\,
\va\!\left(\frac{\al_{11}x+\al_{12}}{\al_{21}x+\al_{22}}\right)\,
\sum\nolimits_{i=0}^3\ga_i x^i\,.
\end{aligned}
\end{equation*}
A $K$-isomorphism of $\vF$ onto $F$ determines the transformation
coefficients uniquely up to the $\G_m$-action defined for each
$c \in \G_m(K)=K^*$ by the assignment
{\small
\[
(\al_{11},\al_{12},\al_{21},\al_{22},\be,\ga_0,\ga_1,\ga_2,\ga_3)
\longmapsto
(c\,\al_{11},c\,\al_{12},c\,\al_{21},c\,\al_{22},c^3\be,c^3\ga_0,c^3\ga_1,c^3\ga_2,c^3\ga_3).
\]
}
If $p\neq 2$ then by completing the square we can normalize $a(x)=0$\,,
and the freedom to transform is restricted by the conditions
$\ga_0=\dots=\ga_3=0$\,.

We will always assume that the function field $F|K$ is separable,
that is, it admits a separating variable. This means that
$x$ or $y$ is a separating variable, that is, $p\neq 2$\, or\,
$a(x)\neq 0$\, or \,$b'(x)\neq 0$\,. Then the condition that the base
field $K$ is algebraically closed in $F$ means that the minimal polynomial
\[
f(X,Y) := Y^2 + a(X)Y + b(X) \in K[X,Y]
\]
is absolutely irreducible, that is, it remains irreducible over the algebraic
closure $\oK$ of $K$\,. As $f$ is monic of degree $2$ in $Y$\,, this means
that there does not exist a polynomial $c(X) \in \oK[X]$ such that
$f(X,c(X))=0$\,. In particular, if $a(X)=0$ then $b(X)$ is not a square in
$\oK[X]$. As $f$ is absolutely irreducible, we can consider the base field
extension
\[
F{\cdot}\oK := F\otimes_K\oK = \oK(x)[Y]/f(x,Y)\oK(x)[Y]\,,
\]
that is, $F{\cdot}\oK = \oK(x,y)$\, where $x$ is transcendental over $\oK$
and \,$f(x,y)=0$\,.

Let $\,\R = \R_{F|K}$\, be the \textit{regular complete model} of the
function field $F|K$\,. It is a regular complete curve over $K$\,,
or more precisely, a geometrically integral regular complete
one-dimensional scheme of finite type over $\Spec(K)$. The algebraic set
$\,R=R_{F|K}\,$ of its closed points consists of the primes $\p$ of $F|K$\,,
whose local rings are the corresponding discrete valuation rings $\O_\p$
of $F|K$\,. The generic point is the only non-closed point of the scheme
$\R$, and its local ring is the function field $F$.

The extended scheme \,$\R\otimes_K\oK := \R\times_{\Spec(K)}\Spec(\oK)$\,
is an integral complete curve over $\oK$\,, whose function field is equal to
\,$F{\cdot}\oK := F\otimes_K\oK$\,. The points of \,$\R\otimes_K\oK$\,
lying over a prime $\p$ of $F|K$ correspond bijectively to the maximal
ideals of the semi-local ring \,$\O_\p{\cdot}\oK = \O_\p\otimes_K\oK$\,,
and their local rings are the corresponding localizations. Though the curve
$\R$ is regular, it may not be smooth, i.e., the extended curve \,$\R\otimes_k\oK$\,
may have singular points (see \cite{Z2}). Recall that a prime $\p$ of
$F|K$ is called \textit{singular}, if the domain \,$\O_\p{\cdot}\oK$ is
not normal, i.e., there is a singular point of the extended curve
\,$\R\otimes_K\oK$\, lying over $\p$\,. Since the arithmetic genus $p_a$
is preserved under base field extensions, 
we have
\[
p_a(\R\otimes_K\oK) = g = 2\,.
\] 
As \,$\R\otimes_K\oK$\, is an integral complete hyperelliptic curve of
arithmetic genus two, the global sections of the tri-canonical divisor
$\C^3$ define an embedding
\[
(1:x:x^2:x^3:y): R\otimes_K\oK \hookrightarrow \PP^4(\oK)
\]
(see \cite[Theorem~2.1]{St2}), and so the extended curve $R\otimes_K\oK$
can be realized as a curve on the cone
\[
S :=
\left\{
(u_0:u_1:u_2:u_3:v)\mid\rank
\left(
\begin{array}{ccc}
u_0 &u_1 &u_2 \\
u_1 &u_2 &u_3
\end{array}
\right)
< 2
\right\}
\]
in the $4$-dimensional projective space $\PP^4(\oK)$.
\bigskip

In the remainder of this section, we invert the preceding considerations.
Given an absolutely irreducible polynomial
\,$f = Y^2 + a(X)Y + b(X) \in K[X,Y]$\,, where $a(X)$ and $b(X)$ are
polynomials of formal degree $3$ and $6$\,, respectively. We consider the
function field
\[
F|K=K(x,y)|K\quad\text{where }\, y^2+a(x)y+b(x)=0\,,
\]
and we assume that it is separable, that is, \,$p\neq 2$\, or \,$a(x)\neq 0$\,
or \,$b'(x)\neq 0$\,.

If $p\neq 2$\,, then it is well known that the genus $g$ of $F|K$ is not
larger than two, and equality holds  if and only if the discriminant
\,$a(x)^2-4\,b(x)$\,  is square-free in \,$K[x]$\, and has degree $5$ or $6$\,.
If \,$p=2$\,, then it is more difficult to determine the genus \,$g$\,.

We will consider a possibly singular projective model of \,$F\oK|\oK$\,
lying on the cone \,$S\subset\PP^4(\oK)$. We note that the cone is the union
of the projective lines
\[
\qquad
L_u :=\{(1:u:u^2:u^3:v)\mid v\in\oK\}\cup\{Q\}\quad\text(u\in\oK)
\]
and
\[
\hspace{-11mm}   
L_\infty := \{(0:0:0:1:v)\mid v\in\oK\}\cup\{Q\}\,,
\]
which have the vertex \,$Q:=(0:0:0:0:1)$\, as their only common point.
The smooth locus \,$S\setminus\{Q\}$\, is described by the atlas
consisting of the two charts
\[
W:= S\setminus L_\infty =
\{(1:u:u^2:u^3:v)\mid(u,v)\in\oK^{\oplus 2}\}
\,\tilde{\longrightarrow}\;\oK^{\oplus 2}
\]
and
\[
\ \uW:= S\setminus L_0\,\, =
\{(\uu^3:\uu^2:\uu:1:\uv)\mid(\uu,\uv)\in\oK^{\oplus 2}\}
\,\tilde{\longrightarrow}\;\oK^{\oplus 2}.
\]
Let \,$C\subset S$\, be the projective integral curve over $\oK$ described in
the first chart $W$ by the minimal equation
\[
y^2 + a(x) y +b(x) = 0 \,,
\]
where the elements $x$ and $y$ of the function field $F$ have been realized
as the rational functions on $C$ that map each point $(1:u:u^2:u^3:v)$ of
\,$C\cap W$\, onto $u$ and $v$\,, respectively. With respect to the second
chart $\uW$ we have the local coordinate functions
\[
\ux:=x^{-1} \quad \text{and} \quad \uy:=x^{-3}y \,
\]
which satisfy the minimal equation
\[
\uy^2+(a_0\,\ux^3+a_1\,\ux^2+a_2\,\ux+a_3)\uy+b_0\,\ux^6+b_1\,\ux^5+\dots+b_6=0\,.
\]
Without using charts, the curve $C$ can be defined as the intersection
of the cone $S$ and the quadratic hypersurface cut out by the equation
\[
v^2+\sum_{i=0}^3 a_i u_i v + \sum_{i=0}^3 b_{2i}u_i^2
+\sum_{i=0}^2 b_{2i+1}u_i u_{i+1} = 0 \,.
\]
In particular, the vertex $Q$ does not lie on the curve. By calculating
the Hilbert polynomial of the curve \,$C\subset\PP^4$\,
(see \cite[p.\ 196]{RS}) we obtain the arithmetic genus:
\[
p_a(C) = 2\,.
\]
\begin{thm}
\label{B1}
The curve $C$ on the cone $S$ is isomorphic to the extended curve
\,$R_{F|K}\otimes_K\oK$\, if and only if the genus $g$ of the function
field \,$F|K$\, is equal to two.
\end{thm}
\begin{proof}
If \,$R\otimes_K\oK\cong C$\, then a fortiori
\,$p_a(R\otimes_K\oK) = p_a(C)$\,, that is, \,$g=p_a(C)$\, and therefore
\,$g=2$\, by the preceding equation. The opposite direction follows
from the first part of this section.
\end{proof}
By \textit{Hironaka's genus formula}~\cite{Hi} the \textit{geometric genus}
\,$p_g(C)$\, of the curve $C$\,, that is, the genus $\og$ of its function
field \,$F\oK|\oK$\,, is equal to
\[
\og = p_a(C) - \sum\dim(
\widetilde{\O}_{C,P}/\O_{C,P}
)
\] 
where the sum is taken over the singular points $P$ of $C$, and where
\,$\widetilde{\O}_{C,P}$\, denote the normalizations of the local rings
\,$\O_{C,P}$\,.
Applying Hironaka's genus formula to the extended curve \,$R\otimes_K\oK$\,
and localizing, we obtain
\textit{Rosenlicht's genus drop formula}
\[
g-\og\; = \sum \dim(
\widetilde{\O_\p{\cdot}\oK}/\O_\p{\cdot}\oK
)
\]
where $\p$ varies over the singular primes of \,$F|K$\,
(cf.\ \cite[Theorem 11]{Ro}).
To determine the genera $\og$ and $g$\,, we have to compute the dimensions of
\,$\widetilde{\O}_{C,P}/\O_{C,P}$\,
and
\,$\widetilde{\O_\p{\cdot}\oK}/\O_\p{\cdot}\oK$\,,
which are called the \textit{singularity degrees} of the points $P$
and the primes $\p$\,, respectively.

As the curve $C$ lies on the punctured cone
\,$S\setminus\{Q\}=W\cup\uW$\,, which via the two charts is locally isomorphic
to the affine plane, the singular points of the curve can be computed by the
Jacobian criterion, and their singularity degrees can be determined by a
finite number of blowups.

As the genus $g$ is preserved under separable base field extensions,
in order to determine $g$\,, we may
assume that the base field $K$ is separably closed. In this case the primes
of \,$F|K$\, correspond bijectively to the primes of \,$F\oK|\oK$\, and hence
to the branches of the curve $C$\,. A branch of $C$ that corresponds to a
singular prime of \,$F|K$\, is necessarily a singular branch and therefore
centered at a singular point of $C$\,. As the minimal equations in the two
charts are monic of degree $2$ in $y$ and $\uy$\,, such a singular point is
necessarily unibranch of multiplicity $2$\,. The singularity degrees of the
primes of \,$F|K$\, can be determined by an algorithm developed in \cite{BS}.

\section{Geometrically elliptic function fields of genus $2$ in characteristic $2$}
\label{C}
\noindent
Let $F|K$ be a one-dimensional separable function field of positive characteristic
$p$\,. As $F|K$ is separable, its \textit{Frobenius pullback}
\[
F_1|K := F^p K|K
\]
is the only subfield of \,$F|K$ such that the extension $F|F_1$\, is inseparable
of degree $p$\,. On the other hand, the Frobenius pullback can be realized as a
base field extension of $F|K$\,, or more precisely,
\[  
F_1|K \cong FK^{\frac{1}{p}}|K^{\frac{1}{p}}.
\] 
In particular, as the genus does not increase under base field extensions, we obtain
\[
g \geq g_1 \geq \og
\]
where $g_1$ denotes the genus of the Frobenius pullback \,$F_1|K$\,.

The function field \,$F|K$\, is called \textit{conservative} if its genus $g$
is equal to the genus $\og$ of \,$F\oK|\oK$\,. By Rosenlicht's genus drop
formula it is non-conservative (that is, $\og < g$) if and only if it admits
a singular prime. The function field $F|K$ is called \textit{geometrically
elliptic} (resp., \textit{geometrically rational}) if \,$F\oK|\oK$\, is
elliptic (resp., rational), that is, \,$\og =1$\, (resp., \,$\og =0$).

The genus-$2$ function field, written in the normal form
\,$y^2 + a(x)y + b(x) = 0$\, of Section~\ref{B}, is called of
\textit{separable type} if it is separable over its canonical quadratic
rational subfield \,$K(x)$\,, that is, \,$p\neq 2$\, or
\,$a(x)\neq 0$\,. 

\begin{thm}
\label{C1}
A one-dimensional separable function field \,$F|K$\, of genus \,$g=2$\,
in characteristic \,$p=2$\, is geometrically elliptic if and only if it
is non-conservative and of separable type.
Such a function field can be put into the normal form
\[
y^2+(a_2\,x^2 +a_0)y+b_6\,x^6 +b_4\,x^4 +b_0 = 0
\]
where \,$a_2, a_0, b_6, b_4, b_0 \in K$\, and
\[
\Delta:=b_6^2\,(a_2^6\,b_0+a_0^2\,a_2^4\,b_4+a_0^3\,a_2^3\,b_6 +a_0^4\,b_6^2)\neq 0\,.
\]
The modular invariant $\oj$ of the elliptic function field \,$F\oK|\oK$\, 
is equal to
\[
\oj\,=(j_1)^{\frac{1}{2}}\qquad \text{where}\qquad j_1 =\frac{a_2^{12}}{\Delta}
\]
is the modular invariant of the Frobenius pullback \,$F_1|K=K(x^2,y)$.

Conversely, if \,$a_2,\, a_0,\, b_6,\, b_4\ \text{and}\ b_0$\, are elements
of the base field $K$ satisfying \,$\Delta\neq 0$\,, then the above polynomial
equation defines a geometrically elliptic function field. Its genus is equal
to two, with the only exceptions that either
\,$j_1\in(K^*)^2$\, and \,$a_0 a_2\in K^2$\,, or
\,$j_1 =0$\, and \,$a_0b_6\in K^2$. 
\end{thm}

In proving the theorem we will also decide when two of these function
fields are isomorphic. We start the proof by considering a separable genus-$2$
function field \,$F|K=K(x,y)|K$\, given in the normal form
\[
\textstyle{
y^2+a(x)y+b(x)=0\quad\text{where }\; a(x)=\sum\nolimits_{i=0}^3 a_i\,x^i
\;\text{ and }\; b(x)=\sum\nolimits_{i=0}^6 b_i\,x^i \,.
}
\]
As \,$p=2$\, the separability of \,$F|K$\, means that \,$a(x)\neq 0$\, or
\,$b'(x)\neq 0$\,. If \,$F|K$\, is of inseparable type, i.e., \,$a(x)=0$\,,
then \,$F_1 =K(x^2,b(x))=K(x),\ \,F\oK =\oK(x,b(x)^{1/2})=\oK(x^{1/2})$\,
and therefore $g_1 = \og =0$\,.

Now we assume that \,$F|K$\, is of separable type, i.e., \,$ a(x)\neq 0$\,.
We further assume that \,$F|K$\, is non-conservative, that is, \,$\og <2$\,.
Let $K'$ be the separable closure of $K$ in $\oK$\,. As the genus is preserved
under separable base field extensions, the function field \,$F K'|K'$\,
is also non-conservative, i.e., it admits a singular prime \,$\p'$\,.
Let $\op$ be the unique prime of \,$F\oK|\oK$\, lying over $\p'$\, and let
$P$ be the corresponding singular point of the curve $C$ on the cone
\,$S\subset\PP^4(\oK)$.

We first assume that $P$ does not lie on the line $L_\infty$\,, that is,
$P\in W$\,. Then by the Jacobian criterion, the coordinates
\,$\ox=x(\op)$\, and \,$\oy=y(\op)$\, of $P$ in the first chart satisfy
\[
a'(\ox)\,\oy + b'(\ox) = 0\,,\;\; a(\ox)=0\; \text{ and }\;
\oy^2=b(\ox)\,. 
\]
Moreover, as by Section~\ref{B} the point $P$ is unibranch, we deduce
\,$a'(\ox)=0$\, and therefore \,$b'(\ox)=0$\,. Thus, $\ox$ is a zero of
\,$a(x)$\, of order larger than one.

If \,$P\in L_\infty$\, then a similar reasoning in the second chart shows
that the polynomial \,$a(x)$\, of formal degree $3$ has order larger than
one at \,$\ox=x(\op)=\infty$\,, that is, \,$\deg a(x)\leq 1$\,.
Moreover, \,$b_5 = 0$\, by analogy with the equation \,$b'(\ox) =0$\,
of the previous case.

In both cases the point $P$ is the only point of the curve $C$ lying on
the line \,$L_{\ox}$\,. As the polynomial \,$a(x)$\, of formal degree $3$
can only admit one multiple zero, we deduce that \,$\p'$\, is the only
singular prime of \,$FK'|K'$\,. In particular, denoting by $\p$ the prime
of \,$F|K$\, lying below $\p'$\,, we conclude that $\p$ is the only
singular prime of \,$F|K$\,.

We will first assume that the multiple zero $\ox$ of \,$a(x)$\, has  order
two. Replacing, if necessary, $x$ and $y$ by \,$x^{-1}$\, and \,$x^{-3}\,y$\,,
respectively, we can assume that \,$\ox\neq\infty$\,, that is, \,$\ox\in\oK$\,.
If \,$\deg a_3(x) =3$\,, then as \,$p=2$\,, by Vieta's formula 
\,$a_1 /a_0$\, is a simple root of \,$a(x)$\, belonging to the base field
$K$\,. Hence, replacing $x$ and $y$ by \,$(x\,-\,{a_1}/{a_0})^{-1}$\, and
\,$(x\,-\,{a_1}/{a_0})^{-3}\, y$\,, respectively, we can arrange that
\,$\deg a_2(x) = 2$\,, that is, \,$a(x)=a_2\,x^2 +a_0$\, and \,$a_2 \neq 0$\,.

If $\ox$ is a triple zero of \,$a(x)$\, and \,$\ox\neq\infty$\,, then
\,$\ox = a_1/a_0$\,, and so by transforming as above we can arrange that
$\infty$ is a triple zero of \,$a(x)$\,, that is, \,$\deg a(x) = 0$\,. 
Thus in each of the two cases we can normalize
\[
a(x) = a_2\, x^2 + a_0 \neq 0 \,.
\] 
The transformations that preserve this normalization preserve the line
\,$L_\infty$\,, and so by Section~\ref{B} they are just of the form
\[
\textstyle{
(x,y)\longmapsto (\al\,x+\de\,,\,\be\,y+\sum\limits_{i=0}^3\ga_i\,x^i)
}
\]
where \,$\al,\,\be\in K^*$\, and \,$\ga_0,\,\ga_1,\,\ga_2,\,\ga_3,\,\de\in K$\,.
To make further normalizations, we are just allowed to substitute
\[
a(x) \longmapsto \be^{-1}\, a(\al\,x+\de)
\]
and
\[
\textstyle{
b(x)\longmapsto\be^{-2}\left(b(\al\,x+\de)
+\sum\limits_{i=0}^3\ga_i^2\,x^{2i}
+a(\al\,x+\de)\,\sum\limits_{i=0}^3 \ga_i\, x^i \right).
}
\]
If \,$a_2\neq 0$\, then we can normalize \,$b_5 = b_3 = b_2 =0$\,, and
furthermore we get \,$b_1 = b'(\ox)=0$\,. If \,$a_2 =0$\, then \,$b_5=0$\,
and we can normalize \,$b_5 = b_3 =b_2 = 0$\,. Thus in both cases we have
\[
b(x) = b_6\,x^6 + b_4\,x^4 + b_0\,,
\]
and the freedom to transform is restricted by the conditions
\[
\ga_1 = \ga_3 = 0 \quad \text{and} \quad
a_2\,\ga_0+(a_2\,\de^2 + a_0)\,\al^{-2}\,\ga_2 = \de^4\,b_6\,.
\]
The coefficients of the minimal equation transform as follows:
\[
\begin{array}{l}
a_2 \longmapsto \al^2\,\be^{-1}\,a_2 \\ \vspace{1mm}
a_0 \longmapsto \be^{-1}\,(a_0 + a_2\,\de^2)\\ \vspace{1mm}
b_6 \longmapsto \al^6\,\be^{-2}\,b_6\\ \vspace{1mm}
b_4 \longmapsto \al^4\,\be^{-2}\,
   (b_4 +\al^{-4}\,\ga_2^2 + \al^{-2}\,\ga_2\,a_2+\de^2\,b_6)\\ \vspace{1mm}
b_0 \longmapsto \be^{-2}\,(b_0 + \ga_0^2 + \ga_0\,a_2\,\de^2
   +\ga_0\,a_0 + b_4\,\de^4 + b_6\,\de^6)\,.
\end{array}
\]
In particular, the class \,$b_6\mod (K^*)^2$\, is an invariant of the function field
\,$F|K$\,. Moreover, if \,$a_2\neq 0$\, (resp., \,$a_2 =0)$, then we can normalize
\,$a_2 =1$\, (resp., \,$a_0 =1)$, and the freedom to transform is furthermore
restricted by the condition \,$\be = \al^2$\, (resp., \,$\be =1)$. Allowing a
quadratic base field extension if necessary, we could also normalize
\,$b_4 =0$\, (resp., \,$b_0=0)$.  

As a hyperelliptic function fields admits exactly one quadratic subfield of
genus zero (see~\cite[Chapter\ \textrm{IV}, Theorem 9]{Ch}) and as
\,$F|K$\, is of separable type, we conclude that the genus \,$g_1$\, of the
Frobenius pullback \,$F_1|K=F^2 K|K$\, is different from zero. As
\,$F_1 =K(x^2,y^2)$\, and \,$a(x)\neq 0$\,, we have \,$y\in F_1$\, and
therefore \,$F_1 =K(z,y)$\, where \,$z:=x^2$\, and
\[
y^2+(a_0+a_2\,z)\,y + b_6\,z^3 + b_4\,z^2 + b_0 = 0 \,.
\]
We notice that $b_6\neq 0$\,, because otherwise by the Jacobian criterion
\,$F_1|K$\, would be the function field of a plane projective smooth conic
curve (respectively, of a projective line) if \,$a_0^2\,b_4\neq a_1^2\,b_0$\,
(respectively, \,$a_0^2\,b_4 = a_1^2\,b_0$), in contradiction with \,$g_1\neq 0$\,.
Moreover, we have \,$\Delta\neq 0$\,, because otherwise \,$F_1|K$\, would be
rational as the function field of a plane projective geometrically integral
cubic curve with a rational non-smooth point. Thus \,$F_1|K$\, is an elliptic
function field and therefore \,$g_1 = \og = 1$\,.

Let $j_1$\, be the modular invariant of \,$F_1|K$\, as introduced in
characteristic two by Tate~\cite{T2}. To compute \,$j_1$\, we replace \,$x$\,
and \,$y$\, by \,$b_6\,x$\, and \,$b_6\,y$\,, respectively, in order to get a
minimal equation that is monic in the two coordinate functions, and then we
obtain \,$j_1=a_2^{12}/\Delta$\, from Tate's formul{\ae}. As the Frobenius map
provides an isomorphism between the function fields
\,$FK^{1/2}|K^{1/2}$\,
and \,$F_1|K$\,, we conclude that the elliptic function fields
\,$FK^{1/2}|K^{1/2}$\, and \,$F\oK|\oK$\, have the invariant \,$j_1^{1/2}$\,.

To prove the last part of the theorem, let be given a polynomial
\[
f(X,Y)=Y^2+(a_2\,X^2+a_0)\,Y+b_6\,X^6+b_4\,X^4+b_0 \in K[X,Y]
\]
whose coefficients satisfy \,$\Delta\neq 0$\,. Then \,$b_6\neq 0$\, and
\,$(a_0,a_2)\neq(0,0)$\,, and this implies that \,$f(X,Y)$\, is absolutely
irreducible. Indeed, if there would exist a polynomial \,$c(X)\in\oK[X]$\,
such that \,$f(x,c(X))=0$\,, then as \,$b_6\neq 0$\, its degree would be
equal to $3$ and by comparing the terms of degree $3$ and $5$ we would get
the contradiction \,$a_0=a_2=0$\,.

Let \,$F|K = K(x,y)|K$\, be the separable function field given by the
absolutely irreducible equation \,$f(x,y)=0$\,. As \,$\Delta\neq 0$\,,
the Frobenius pullback \,$F_1|K = K(x^2,y^2)|K = K(x^2,y)|K$\, is an
elliptic function field, and so \,$g_1 = \og =1$\,.

To express the genus $g$ of \,$F|K$\, in terms of the coefficients, we will
apply Rosenlicht's genus drop formula. To determine the singularity degree
of a prime $\p$ of \,$F|K$\,, we look for a natural number $n$ such that
the restriction \,$\p_n$\, of $\p$ to the $n$-th Frobenius pullback
\,$F_n := F^{2^n}K|K$\, is rational. Such an integer exists if $K$
is separably closed (see \cite[Lemma 2.1]{BS}). We write the $2^n$-power
of the separating variable $x$ as a Laurent series in a local parameter
at \,$\p_n$\,. Then the singularity degree of $\p$\,, as well as other
properties of the local ring $\O_\p$ can be determined from
this Laurent series expansion (see~\cite{BS}). To finish the proof of the
theorem, we divide the discussion into three cases, according to
\,$a_2=0$\,, \,$a_0/a_2\in K^2$\, and \,$a_0/a_2\in K\setminus K^2$\,,
and provide a direct proof of the following corollary.
\begin{cor}
\label{C2}
A function field over a field $K$ of characteristic $2$
is geometrically elliptic of genus $2$ if and only if it can be put
into one of the three normal forms:
\[
\begin{array}{rl}    
(i)  & y^2+y+b_6\,x^6+b_4\,x^4+b_0=0\quad\text{where}\;\; b_6\notin K^2 \vspace{1mm}\\
(ii) & y^2+x^2\,y+b_6\,x^6+b_4\,x^4+b_0=0\quad
       \text{where }\,b_0\notin K^2\,\text{ and }\,b_6\neq 0 \vspace{1mm}\\
(iii)& y^2+(x^2+a_0)\,y+b_6\,x^6+b_4\,x^4+b_0=0 \quad
               \text{where }\, a_0\notin K^2\,\text{ and }\,\Delta \neq 0\,.
\end{array}
\]
The first case happens if and only if \,$j_1=0$\,.
\end{cor}
\begin{proof}
(i) \ \ We assume that \,$a_2=0$\,, and normalize \,$a_0=1$\,.
Let \,$\p$\, be a singular prime of \,$F|K$\,. As the only singular point of the curve \,$C$\,
lies on the line \,$L_{\infty}$\,, we conclude that $\p$ is a pole of $x$.  Thus, the restriction  \,$\p_1$\, of
\,$\p$\, to the Frobenius pullback \,$F_1|K=K(x^2,y)|K$\, is the only pole of \,$z=x^2$. Hence 
\,$\deg(\p_1)=1,\,\ \ord_{\p_1}(z)=-2,\,\ \ord_{\p_1}(y)=-3$\, and \,$t:=z/y$\, is a local parameter at
 \,$\p_1$\,. To write \,$z=x^2$\, as a Laurent series in $t$, we notice that
\[
t^{-2}z^2+t^{-1}z=b_6z^3+b_4z^2+b_0
\]    
and by comparing successively coefficients we obtain
\[
z=b_6^{-1}(t^{-2}+b_4t^0+b_6t^1+\cdots)
\]    
where the dots stand for terms of order larger than $1$. Now we can apply \cite[Proposition 4.1]{BS}:
 If \,$b_6\notin K^2$\, (resp., \,$b_6\in K^2$) then the singularity degree of $\p$ is
 equal to one (resp., zero) and so by Rosenlicht's genus drop formula \,$g=2$\, (resp., \,$g=1$).

(ii)\ \  We assume that \,$a_2\neq 0$\, and \,$a_0/a_2\in K^2$\,, and normalize \,$a_2=1$\, and \,$a_0=0$\,.
 Let $\p$ be a singular prime of \,$F|K$\,. As the only singular point of the curve \,$C\subset \PP^4\,$
 lies on the line \,$L_0$\,, we conclude that $\p$ is a zero of \,$x$\,. Thus the restriction \,$\p_1$\, of \,$\p$\,
to \,$F_1|K=K(x^2,z)|K$\, is the only zero of \,$z=x^2$\,.

We first assume that \,$b_0$\, is a square, say \,$b_0=c^2$\, where \,$c\in K^*$. 
Then \,$F_1=K(t,z)$\,
 where \,$t:=y+c$\, and
\[
t^2+cz+tz+b_4z^2+b_6z^3=0\,.
\]
The prime \,$\p_1$\, is centered at the smooth rational point \,$(\ol{t},\ol{z})=(0,0)$\, of the affine
 plane cubic curve, and therefore \,$\deg(\p_1)=1$\,. Moreover, $t$ is a local parameter at
 \,$\p_1$\,. Expanding \,$z=c^{-1}t^2+c^{-2}t^3+\cdots$\, we deduce from \cite[Proposition 4.1]{BS}
 that the prime $\p$ is non-singular and therefore \,$g=1$\,.

Now we assume that \,$b_0\notin K^2$. Applying the preceding considerations to \,$F_2|K$\, 
instead of \,$F_1|K$\,, we conclude that the prime \,$\p_2$\, is rational, \,$t:=y^2+b_0$\, is a local
 parameter at $\,\p_2$\,, and $\,x^4=b_0^{-2}t^2+b_0^{-4}t^3+\cdots$\,. Then it follows from \cite[Proposition 4.3]{BS} 
that \,$\p|\p_1$\, is ramified and the singularity degree of $\p$ is equal to $1$\,, and therefore \,$g=2$\,.

(iii)\ \  We assume that \,$a_2\neq 0$\, and \,$a_0/a_2\notin K^2$\,, normalize \,$a_2=1\,,$ and so we have
 \,$a_0\notin K^2$\,. Let $\p$ be a singular prime of \,$F|K$\,. As the only singular point of $C$ lies on 
the line \,$L_{a_0^{1/2}}$\,, we conclude that $\p$ lies over the $(x^2+a_0)$-adic prime of the quadratic
 rational subfield $K(x)$ of $F|K$. Denoting by $\oy$ the residue class of $y \mod \p$, we have 
\[
\oy=(b_6a_0^3+b_4a_0^2+b_0)^{1/2}=b_6^{-1}j_1^{-1/2}+a_0^2b_6.
\]
We will first assume that \,$\oy\in K$\,, that is, \,$j_1\in (K^*)^2$\,. Then \,$\p_1$\, is a rational prime
 of \,$F_1|K=K(x^2,y)|K$\,, and \,$t:=y+\oy$\, is a local parameter at \,$\p_1$\,. To write \,$z=x^2$\,
 as a Laurent series in $t$, we enter into the polynomial equation
\[
(t+\oy)^2+(z+a_0)(t+\oy)+b_6z^3+b_4z^2+b_0=0
\]    
and obtain 
\[
x^2=a_0+b_6j_1^{1/2}t^2+b_6^2j_1t^3+\cdots.
\]
It now follows from \cite[Proposition 4.1]{BS} that the singularity degree of $\p$ is equal to one,
and therefore \,$g=2$\,.

Now we assume that \,$\oy\notin K$\,, that is, \,$j_1\notin K^2$\,. As \,$F_1=K(x^2,y)$\, we obtain
\[
F_2=K(x^4,y^2)=K(x^4,t)\ \text{ where } t:=y^2+\oy^2
\]
is a local parameter at the rational prime \,$\p_2$\,. From the polynomial equation 
\[
(t+\oy^2)^2+(x^4+a_0^2)(t+\oy^2)+b_6^2x^{12}+b_4^2x^8+b_0^2=0
\]
we get the power series expansion
\[
x^4=a_0^2+b_6^2j_1t^2+b_6^4j_1^2t^3+\cdots.
\]
By \cite[Proposition 4.1]{BS} the residue field of \,$\p_1$\, is equal to \,$K(j_1^{1/2})$\,.

If \,$a_0\notin K^2(j_1)$\, then by \cite[Proposition 4.3]{BS} the residue field of $\p$ is equal to
 \,$K(j_1^{1/2}, a_0^{1/2})$\,, the singularity degree of $\p$ is equal to $1$, and therefore \,$g=2$\,.

If \,$a_0\in K^2(j_1)$\,, that is, \,$a_0\in K^2(\oy^2)$\, say \,$a_0=\al^2+\be^2\oy^2$\, where \,$\al, \be\in K$\,, then defining \,$w:=x+\al+\be y$\,, we get the expansion
\[
w^4=(b_6^2j_1+\be^4)t^2+b_6^4j_1^2t^3+\cdots\,,
\] 
and so by \cite[Proposition 4.3]{BS} \,$\p|\p_1$\, is ramified, the singularity degree 
of \,$\p$\, is equal to 1, and therefore \,$g=2$\,.
\end{proof}

\section{Genus-2 function fields of inseparable type}
\label{D}
\noindent
Let \,$F|K$\, be a one-dimensional separable function field of genus $2$, written in the
 normal form \,$y^2+a(x)y+b(x)=0$\, of Section~\ref{B}. We assume that \,$F|K$\, is 
of \textit{inseparable type} or, equivalently, it is an inseparable extension of its canonical
 quadratic rational subfield \,$K(x)$\,, that is, \,$p=2$\, and \,$a(x)=0$\,. Therefore
\[
y^2=b(x)=\sum_{i=0}^6b_ix^i\in K[x] \ \text{ and }\ b'(x)=b_5x^4+b_3x^2+b_1\neq 0\,. 
\] 
By Section~\ref{B} the polynomial \,$b(x)$\, is uniquely determined by the
isomorphism class of \,$F|K$\, up  to the substitutions 
\[
b(x)\mapsto \be^{-2}\left(\al_{21}x+\al_{22}\right)^6\,\,
b\!\left(\frac{\al_{11}x+\al_{12}}{\al_{21}x+\al_{22}}\right)
  + \sum\nolimits_{i=0}^3 \ga_i^2 x^{2i} 
\]
where \,$(\al_{ij})\in\operatorname{GL}_2(K),\ \,\be\in K^*$\, and \,$\ga_0, \ga_1, \ga_2, \ga_3\in K$\,. In particular,
\[
b'(x)\mapsto \be^{-2}\det(\al_{ij})\left(\al_{21}x+\al_{22}\right)^4\,\,
b'\!\left(\frac{\al_{11}x+\al_{12}}{\al_{21}x+\al_{22}}\right).
\]
Replacing if necessary $x$ by $\frac{1}{x}$ or $1+\frac{1}{x}$, we can 
arrange that \,$b_5\neq 0$\,, and so we can normalize \,$b_5=1\,$.

\begin{thm}
\label{D1}
A one-dimensional separable function field of genus \,$g=2$\,
in characteristic \,$p=2$\, is geometrically rational if and only if it
is of inseparable type,
that is, it can be put into the normal form
\[
y^2=b(x)=\sum_{i=0}^6b_ix^i\quad  \text{where }\, b_5=1\, \text{ and }\, b_0,b_1,b_2, b_3, b_4, b_6\in K\,. 
\]
If \,$b_3\neq 0$\, and if (after an eventual quadratic separable base field extension) the two
 roots of the polynomial \,$b'(T^{1/2})=T^2+b_3T+b_1$\, belongs to the base field \,$K$\,, then
 the function field of inseparable type has genus two if and only if each such root $c$
 satisfies \,$c\notin K^2$\, or \,$b_0+b_2c+b_4c^2+b_6c^3\notin K^2$\,.

If \,$b_3=0$\,, then the genus is equal to two if and only if one of the following three cases occurs:

\begin{enumerate}
\item[(i)]\ $b_1\in K\setminus K^2$ \smallskip
\item[(ii)]\ $b_1\in K^2\setminus K^4$\,, and \,$b_0+b_1b_4\notin K^2$\, or \,$b_2+b_1b_6\notin K^2$ \smallskip
\item[(iii)]\ $b_1\in K^4$\, and \,$\sum\nolimits_{i=0}^3 b_{2i}b_1^{i/2}\notin K^2$\,.   
\end{enumerate}
\end{thm} 
\begin{proof}
The first part of the theorem follows from the first part of Theorem \ref{C1}.

Let \,$F|K$\, be the function field given by the equation \,$y^2=b(x)$\,. As \,$b'(x)\neq 0$\,,
 the $n$-th Frobenius pullback is equal to 
\[
F_n=F^{2^n}K=K(x^{2^n},y^{2^n})=K(x^{2^n},b(x)^{2^{n-1}})=K(x^{2^{n-1}}) 
\]
for each natural number $n$. If $\p$ is a singular prime of \,$F|K$\,, then by the Jacobian
 criterion it is necessarily a zero of \,$b'(x)=x^4+b_3x^2+b_1$\,. 

We will first assume that \,$b_3\neq 0$\,, and that we can factorize
\[
b'(x)=(x^2+c)(x^2+d)\, \text{ where }\, c,d\in K\, \text{ and }\, c\neq d\,. 
\]
Let \,$\p\in \R_{F|K}$\, be the zero of \,$x^2+c$\,. If \,$c\in K^2$\, then \,$\p_1$\, is a rational prime, \,$t:=x+c^{1/2}$\,
 is a local parameter at \,$\p_1$\,, 
\[
y^2=(b_0+b_2c+b_4c^2+b_6c^3)t^0+\big(b_2+c^{1/2}(c+d)\big)t^2+(c+d)t^3+\cdots,
\]
and it follows from \cite[Proposition 4.1]{BS} that the singularity degree of $\p$ is equal to $1$
 (respectively, $0$) if \,$b_0+b_2c+b_4c^2+b_6c^3$\, does not belong (respectively, belongs)
 to $K^2$.

Now we assume that \,$c\notin K^2$. Then \,$\p_1$\, is the $(x^2+c)$-adic prime of \,$F_1=K(x)$,
 \,$\p_2$\, is rational, \,$t:=x^2+c$\, is a local parameter at \,$\p_2$\,, and 
\[
y^4=(b_0+b_2c+b_4c^2+b_6c^3)^2t^0+\big(b_2^2+c(c+d)^2\big)t^2+(c+d)^2t^3+\cdots.
\]
The residue class of \,$y\! \mod \p$ is equal to \,$\oy=(b_0+b_2c+b_4c^2+b_6c^3)^{1/2}$\,.

If \,$\oy\notin K(c^{1/2})$\, then \,$\p|\p_1$\, is inertial and by \cite[Theorem 3.2]{BS} the singularity
degree of $\p$ is equal to $1$.

If \,$\oy\in K(c^{1/2})$\, say \,$\oy=\al+\be c^{1/2}$\, where \,$\al, \be\in K$\,, and if \,$z:=y+\al+\be x$\, then
\[
z^4=(c(c+d)^2+b_0^2+\be^4)t^2+(c+d)^2t^3+\cdots,
\]
and, as \,$c\notin K^2$\,, \,$c\neq d$\, and hence the coefficient of \,$t^2$\, is non-zero, we deduce that
 \,$\p|\p_1$\, is ramified, \,$\ord_{\p}(z)=1$\, and by \cite[Theorem 3.2]{BS} we again conclude that the 
singularity degree of $\p$ is equal to $1$.

Thus the singularity degree of the two zeros of \,$b'(x)$\, are not larger than one, 
and so by Rosenlicht's genus drop formula the genus $g$ is equal to two if and only if
the two singularity degrees are equal to one.

Now we assume that \,$b_3=0$\,. Let \,$\p\in\R_{F|K}$\, be the zero of \,$b'(x)=x^4+b_1$\,.

(i)\ \ We assume that $\,b_1\in K\setminus K^2$. Then \,$\p_1$\, is the $(x^4+b_1)$-adic prime
 of \,$F_1=K(x)$\,, \,$\p_3$\, is rational, \,$t:=x^4+b_1$\, is a local parameter at \,$\p_3$\,, and
\[
y^8=(b_0^2+b_2^2b_1+b_4^2b_1^2+b_6^2b_1^3)^2t^0+(b_2+b_1b_6)^4t^2+(b_1+b_4^4+b_6^4b_1^2)t^4+t^5+b_6^4t^6\,.
\]
The residue class of \,$y\! \mod \p$\, is equal to \,$\oy=(b_0^2+b_2^2b_1+b_4^2b_1^2+b_6^2b_1^3)^{1/4}$\,.

If \,$\oy\notin K(b_1^{1/4})$\, then \,$\p|\p_1$\, is inertial and by \cite[Theorem 3.2]{BS} the singularity degree of $\p$ is equal to two.

Now we assume that \,$\oy\in K(b_1^{1/4})$\, say \,$\oy^4=\al^4+\be^4b_1+\ga^4b_1^2+\de^4b_1^3$\,
 where \,$\al, \be, \ga, \de\in K$\,. Then \,$(b_0+b_4b_1+\al^2+\ga^2b_1)^2=b_1(b_2+b_6b_1+\be^2+\de^2b_1)^2$\,. As \,$b_1\notin K^2$\,, 
this means 
\[
b_0=b_4b_1+\al^2+\ga^2b_1\ \text{ and }\ b_2=b_6b_1+\be^2+\de^2b_1\,. 
\] 
Defining \,$z:=y+\al+\be x+\ga x^2+\de x^3$\, we obtain
\[
z^8=(b_1+b_4^4+b_6^4b_1^2+\ga^8b_1^2)t^4+t^5+(b_6^4+\de^8)t^6\,.
\]
As \,$b_1$\, is not a square, the coefficient of \,$t^4$\, is non-zero, hence \,$\p|\p_1$\, is ramified, $z$ is a
 local parameter at $\p$, and by \cite[Theorem 3.2]{BS} the singularity degree of $\p$ is again 
equal to two.

(ii)\ \  We assume that \,$b_1\in K^2\setminus K^4$\, say \,$b_1=c^2$\, where \,$c\in K\setminus K^2\,$. Then \,$\p_1$\, is
 the $(x^2+c)$-adic prime of \,$F_1=K(x)$\,, \,$\p_2$\, is rational, \,$t:=x^2+c$\, is a local parameter
 at \,$\p_2$\,, and 
\[
y^4=(b_0+b_2c+b_4^2c^2+b_6c^3)^2t^0+(b_2+b_6c^2)^2t^2+(c+b_4^2+b_6^2c^2)t^4+t^5+b_6^2t^6\,.
\]
The residue class of \,$y\! \mod \p$\, is equal to \,$\oy=(b_0+b_2c+b_4c^2+b_6c^3)^{1/2}$\,.

If \,$\oy\notin K(c^{1/2})$\, then \,$\p|\p_1$\, is inertial and by \cite[Theorem 3.2]{BS} the singularity degree of $\p$ is equal to two.

Now we assume that \,$\oy\in K(c^{1/2})$\, say \,$\oy^2=\al^2+\be^2c$\,, i.e., \,$b_0+b_4c^2+\al^2=(b_2+b_6c^2+\be^2)c$\,
 where \,$\al, \be\in K$\,. Defining \,$z:=y+\al+\be x$\, we obtain 
\[
z^4=(b_2+b_6c^2+\be^2)^2t^2+(c+b_4^2+b_6^2c^2)t^4+t^5+b_6^2t^6\,.
\]
If \,$b_2\neq b_6c^2+\be^2$\,, then \,$\p|\p_1$\, is ramified, $z$ is a local parameter at $\p$, and by \cite[Theorem 3.2]{BS} 
the singularity degree of $\p$ is again equal to two. If \,$b_2=b_6c^2+\be^2$\,, i.e.,
 \,$b_0=b_4c^2+\al^2$\,, then the singularity degree of $\p$ is equal to zero.

(iii)\ \ We assume that \,$b_1\in K^4$\, say \,$b_1=c^4$\, where \,$c\in K$\,. Then \,$\p_1$\, is rational,
 \,$t:=x+c$\, is a local parameter at \,$\p_1$\,, and
\[
y^2=(b_0+b_2c^2+b_4^2c^4+b_6c^6)t^0+(b_2+b_6c^4)t^2+(c+b_4+b_6c^2)t^4+t^5+b_6t^6\,.
\]
Now we can apply \cite[Proposition 4.1]{BS}: If \,$\sum\nolimits_{i=0}^3 b_{2i}c^{2i}\notin K^2$\, then the singularity
 degree of $\p$ is equal to two. If \,$\sum\nolimits_{i=0}^3 b_{2i}c^{2i}\in K^2$\, then the singularity degree of $\p$ is 
equal to one (respectively, zero) if \,$b_2+b_6c^4\notin K^2$\, (respectively, \,$b_2+b_6c^4\in K^2$).            
\end{proof}

\section{Fibrations by non-smooth curves of arithmetic genus $2$
in characteristic $2$.}
\label{E}
\noindent
Let $k$ be an algebraically closed field of characteristic $2$\,, and
\[
S = S(k) :=
\left\{
(u_0:u_1:u_2:u_3:v)\in \PP^4(k) \mid\rank
\left(
\begin{array}{ccc}
u_0 &u_1 &u_2 \\
u_1 &u_2 &u_3
\end{array}
\right)
< 2
\right\}
\]
be the cone that in Section~\ref{B} has been considered over the field $\oK$.
Influenced by Section~\ref{C} we announce:
\begin{thm}
\label{E1}
The algebraic variety
\begin{equation*} 
\begin{aligned}
Z :=   \text{{\Large \{}} 
&
((u_0:u_1:u_2:u_3:v),(a_0,a_2,b_0,b_4,b_6))\in S\times\A^5\mid
\\  & \;\;\;
v^2+(a_0\,u_0 +a_2\,u_2)v+b_0\,u_0^2 +b_4\,u_2^2 +b_6\,u_3^2\, =\, 0\,
\text{{\Large \}}}
\end{aligned}
\end{equation*}
is an irreducible smooth sixfold. The projection morphism
\[
\pi : Z \longrightarrow \A^5
\]
is proper and flat, and its fibres are non-smooth projective curves of
arithmetic genus $2$\,, which do not pass through the vertices.

The fibre over the point \,$(a_0,a_2,b_0,b_4,b_6)\in \A^5(k)$\, 
is a geometrically elliptic curve (i.e., an integral curve of geometric
genus 1) if and only if
\[
\Delta:=
b_6^2\,(a_2^6\,b_0+a_0^2\,a_2^4\,b_4+a_0^3\,a_2^3\,b_6 +a_0^4\,b_6^2)\neq 0\,.
\]
In this case, the fibre has a cusp as its only singularity, and the elliptic
modular invariant of its non-singular projective model is equal to
\,$a_2^6/\Delta^{1/2}$.
\end{thm}
During the proof of the theorem we will describe the singularities
of all fibres, and discuss how they move.
We start the proof by noting that the variety $Z$ is the closed subset of
\,$S\times\A^5$\, contained in the smooth subvariety
\,$(S\setminus\{Q\})\times\A^5$\, that in the charts
\,$W\times\A^5\tilde{\longrightarrow}\A^7\,$\, and
\,$\uW\times\A^5\tilde{\longrightarrow}\A^7\,$\, 
is given by the equations
\[
v^2 +(a_0 +a_2\,u^2)\,v+b_0 +b_4\,u^4+b_6\,u^6 = 0 
\]
and
\[
\uv^2 +(a_0\,\uu^3 +a_2\,\uu)\,\uv +b_0\,\uu^6+b_4\,\uu^2 +b_6 =0\,,
\]
respectively. Hence $Z$ is irreducible, smooth and of dimension $6$\,.
As $S$ is projective, the projection morphism $\pi$ is proper.
The fibration \,$\pi : Z\rightarrow\A^5$\, provides a
$5$-dimensional family of projective curves on the punctured cone
\,$S\setminus\{Q\}$\,. For each point $(a_0,a_2,b_0,b_4,b_6)$ of the base
\,$\A^5$\, the corresponding curve is given in the two charts
\,$W$\, and \,$\uW$\,
by the above equations. As the base $\A^5$ is smooth, as the total space
$Z$ is Cohen-Macaulay as a smooth variety, and as the dimension of each
fibre is equal to \,$\dim(Z)-\dim(\A^5)=1$\,, we conclude that the
morphism $\pi:Z\rightarrow\A^5$\, is flat (see
\cite[Theorem 18.16]{E}), and so the fibres have the same arithmetic genus.

The singular points of the fibres are obtained by applying the Jacobian
criterion to the curves in the two charts of the punctured cone. The types
of the singularities can be read off from the blowup sequences. Each curve
of the family has at the point
\[
\left(a_2^{3/2}:a_0^{1/2}a_2:a_0\,a_2^{1/2}:a_0^{3/2}:
(b_0\,a_2^6+b_4\,a_0^4\,a_2^2+b_6\,a_0^6)^{1/2}\right)
\]
a singularity with double tangent line. It is a cusp if and only if the
second factor of $\Delta$ is non-zero. If the second factor vanishes and
\,$a_2\neq 0$\,,
then it is a tacnode, i.e., a two-branched point of singularity degree two.
If \,$a_2 = 0\,,\ a_0\neq 0\,,\ b_6 = 0\, \text{ and }\, b_4\neq 0$\,, then
it is a ramphoid cusp, i.e., a unibranch point of singularity degree two.
If \,$a_2 = 0\,,\ a_0\neq 0\,,\ b_6 = 0\, \text{ and }\, b_4= 0$\,, then
it is a two-branched point of multiplicity three.

If \,$a_2\neq 0$\, and \,$b_6 =0$\,, then the curve has a second
singularity, namely a node at the point $(0:0:0:1:0)$ on the line
$L_\infty$\,, which collides with the first singularity if \,$a_2$\,
tends to zero. If $(a_0,a_2)\neq (0,0)$,  then there are no other
singularities on the fibre.
If $(a_0,a_2)=(0,0)$, then the fibre is a double smooth rational curve,
and so it is non-reduced and its points are singular.

If $(a_2,a_0)\neq (0,0)$, then the fibre is non-integral if and only if
\,$b_6 =0$\, and \,$a_0^2\,b_4 = a_2^2\,b_0$\,. In this case it is the
union of two smooth rational curves. If \,$a_2\neq 0$\,, then the two
components intersect at the two singular points
with the multiplicities $2$ and $1$\,. 
If \,$a_2=0$\,, then the two components intersect at the only singular
point %
with multiplicity $3$\,.

\medskip
\noindent
\textbf{Remark.}
If by homogenizing we enlarge the base of $\pi$ from $\A^5$ to $\PP^5$,
then the total space acquires singularities at 
\,$\{Q\}\times (\PP^5\setminus \A^5)$\,, and the fibres over
\,$\PP^5\setminus \A^5$\, pass through the vertices.
\begin{thm}\label{E2}
If \,$\f:T\rightarrow B$\, is a proper morphism of irreducible smooth
algebraic varieties such that almost all fibres are geometrically elliptic curves
of arithmetic genus $2$\,, then the fibration \,$\f:T\rightarrow B$\,
is, up to birational equivalence, a base extension of the fibration
\,$\pi:Z\rightarrow \A^5$\,.
\end{thm}
\begin{proof}
By Section~\ref{A}, the one-dimensional function field \,$k(T)|k(B)$\,
is geometrically elliptic of genus $2$\,, and so by Theorem~\ref{C1} it can
be put into the normal form
\[
k(T) = k(B)(x,y)\quad\text{where}\quad
y^2+(a_0+a_2\,x^2)y+b_0 +b_4\,x^4 +b_6\,x^6  = 0\,,
\]
$a_0, a_2, b_0, b_4, b_6 \in k(B)$\, and
\,$\Delta(a_0,a_2,b_0,b_4,b_6)\neq 0$\,.
Let $\vB$ be the closed irreducible affine subvariety of $\A^5$ with the
coordinate algebra \,$k[a_0,a_2,b_0,b_4,b_6]$\,, let
\[
\vT:=\pi^{-1}(\vB)\subseteq Z\subset S\times\A^5 
\]
and let \,$\vf:\vT\rightarrow\vB$\, be the corresponding closed
subfibration of \,$\pi:Z\rightarrow\A^5$\,. By Theorem~\ref{E1}
the morphism $\vf$ is proper and its fibres are non-smooth projective curves of
arithmetic genus $2$\,. As \,$\Delta\neq 0$\,, almost every fibre is
a geometrically elliptic curve with a cusp as its only singularity. By
construction \,$k(\vB)=k(a_0,a_2,b_0,b_4,b_6)\,,\,\ k(\vT)=k(\vB)(x,y)$\,
and therefore
\[
k(T)\,\cong\,k(\vT)\otimes_{k(\vB)} k(B)\,.
\]
By restricting the base $B$ of the fibration \,$\f:T\rightarrow B$\, to a
dense open subset, we can arrange that the rational functions
\,$a_2,\,a_0,\,b_6,\,b_4$\, and \,$b_0$\, become regular on $B$\,, and hence
define a dominant morphism \,$B\rightarrow\vB$\,. Thus we have the
fibre product \,$\vT\times_{\vB}B$\,, whose function field is isomorphic
to \,$k(\vT)\otimes_{k(\vB)} k(B)$\, and hence isomorphic to \,$k(T)$\,.
More precisely, the inclusion \,$k(\vT)\subseteq k(T)$\, induces a rational
map \,$T\dashrightarrow\vT$\, and hence a rational map of $B$-schemes
\,$T\dashrightarrow\vT\times_{\vB}B$\,, which is birational, because
the induced homomorphism between the function fields is the above isomorphism.
\end{proof}

To diminish the dimension of the base of the fibration $\pi$\,, according to
Corollary~\ref{C2} we divide the discussion into the cases \,$\oj \neq 0$\,
and \,$\oj=0$\,, normalize \,$a_2=1$\, and \,$a_0=1$\,, respectively, and
obtain as base varieties the affine spaces $\A^4$\, and \,$\A^3$\,,
respectively. By admitting separable quadratic base extensions if necessary,
we can further normalize \,$b_4=0$\, and \,$b_0=0$\,, respectively, and
diminish the dimensions of the bases by $1$\,.

In the first case we get the irreducible smooth fourfold
\begin{equation*} 
\begin{aligned}
X :=   \text{{\Large \{}} 
&
((u_0:u_1:u_2:u_3:v),(a_0,b_0,b_6))\in S\times\A^3\mid
\\  & \;\;\;
v^2+(a_0\,u_0 + u_2)v+b_0\,u_0^2 + b_6\,u_3^2\, =\, 0\,
\text{{\Large \}}}
\end{aligned}
\end{equation*}
equipped with the proper and flat projection morphism \,$\chi:X\rightarrow\A^3$\,.

If \,$\Delta:=b_6^2\,(b_0+a_0^3\,b_6+a_0^4\,b_6^2)\neq 0$\,,
then the fibre over the point $(a_0,b_0,b_6)$ is geometrically elliptic
with the modular invariant \,$\Delta^{-1}\neq 0$\,, and has a cusp as
its only singularity. This describes the generic behavior of the fibres of $\chi$.

By the discussion following Theorem~\ref{E1}, we also know the structure of the
bad fibres. If \,$b_6\neq0$\, and the second factor of $\Delta$ vanishes,
then the fibre is a rational curve with a tacnode as its only singularity.
If \,$b_6=0$\, and \,$b_0\neq0$\,, then the fibre is a rational curve with
a cusp and a node as its only singularities.
In the remaining case where \,$b_0=b_6=0$\,, the fibre is a union of two
smooth rational curves, which meet in the two singular points with intersection
multiplicities two and one.

In the second case where $\oj=0$\,, we get the irreducible smooth threefold
\[ 
Y :=   \text{{\Large \{}} 
((u_0:u_1:u_2:u_3:v),(b_4,b_6))\in S\times\A^2\mid
v^2+u_0\,v+b_4\,u_2^2 + b_6\,u_3^2\, =\, 0\,
\text{{\Large \}}}
\]
and the proper and flat projection morphism \,$\eta:Z\rightarrow\A^2$\,.
If \,$b_6\neq0$\,, then the fibre over the point $(b_4,b_6)$ is geometrically
elliptic with the modular invariant $0$\,, and has a cusp as its only
singularity.

The bad fibres are described as follows:
If \,$b_6=0$\, and \,$b_4\neq0$\,, then the fibre is a rational curve with a
ramphoid cusp as its only singularity.
If \,$b_4=b_6=0$\,, then the fibre is a union of two smooth
rational curves, which meet in only one point with intersection
multiplicity three.
\begin{cor}
\label{E3}
Let \,$\f:T\rightarrow B$\, be a proper morphism of irreducible smooth
varieties such that almost all fibres are geometrically elliptic curves of
arithmetic genus two. If the modular invariants of the fibres are not
identically zero (respectively, equal to zero), then the fibration
\,$\f:T\rightarrow B$\,, after an eventual separable quadratic
base extension, is birational equivalent to a base extension of the
fibration \,$\chi:X\rightarrow\A^3$\, (respectively,
\,$\eta:Y\rightarrow\A^2)$.
\end{cor}

>From Section~\ref{D} we obtain by similar arguments the following result.
\begin{thm}
\label{E4}
The algebraic variety
\begin{equation*} 
\begin{aligned}
V :=   \text{{\Large \{}} 
&
((u_0:u_1:u_2:u_3:v),(b_0,\dots,b_4,b_6))\in S\times\A^6\mid
\\  & \;\;\;
v^2+b_0\,u_0^2+b_1\,u_0\,u_1+b_2\,u_1^2+b_3\,u_1\,u_2+b_4\,u_2^2+u_2\,u_3+b_6\,u_3^2\,=\,0\,
\text{{\Large \}}}
\end{aligned}
\end{equation*}
is an irreducible smooth sevenfold. The projection morphism
\,$\mu:V\rightarrow\A^6$\, is proper and flat, and its fibres are rational
curves of arithmetic genus two. If \,$b_3\neq0$\, (respectively, \,$b_3=0$), then
the fibre over the point \,$(b_0,\dots,b_4,b_6)\in\A^6(k)$\, has two cusps
(respectively, a ramphoid cusp).

If \,$\f:T\rightarrow B$\, is a proper morphism of irreducible non-smooth
varieties such that almost all fibres are rational curves of arithmetic genus
two, then the fibration \,$\f:T\rightarrow B$\, is, up to birational equivalence,
a base extension of the fibration
\,$\mu:V\rightarrow\A^6$\,.
\end{thm}


\end{document}